\documentclass[11pt, reqno]{amsart}
\usepackage{amsmath,amsthm,amssymb,enumerate,amsxtra}
\usepackage[colorlinks,linkcolor=citation,citecolor=citation,urlcolor=citation]{hyperref}
\usepackage{cleveref}
\usepackage[all]{xy}
\usepackage{array}

\usepackage[dvipsnames]{xcolor}
\definecolor{citation}{rgb}{0,.40,.80}

\newcommand \C {\mathbb C}

\newcommand \F {\mathbb F}

\newcommand \N {\mathbb N}
\renewcommand \O {\mathcal O}
\renewcommand \P {\mathbb P}

\newcommand \Z {\mathbb Z}

\DeclareMathOperator \Aut {Aut}

\DeclareMathOperator \Hom {Hom}

\DeclareMathOperator{\Fix}{Fix}

\newtheorem {thm} {Theorem}[section]

\newtheorem {lemma} [thm] {Lemma}
\newtheorem {prop} [thm] {Proposition}

\theoremstyle{definition}

\newtheorem {cons} [thm] {Construction}

\newtheorem {rmk}[thm] {Remark}

\numberwithin{equation}{section}

\makeatletter
\newcommand{\xleftrightarrow}[2][]{\ext@arrow 3359\leftrightarrowfill@{#1}{#2}}
\newcommand{\dashto}[2][]{\ext@arrow 0359\rightarrowfill@@{#1}{#2}}
\newcommand{\xdashleftarrow}[2][]{\ext@arrow 3095\leftarrowfill@@{#1}{#2}}
\newcommand{\xdashleftrightarrow}[2][]{\ext@arrow 3359\leftrightarrowfill@@{#1}{#2}}
\def\rightarrowfill@@{\arrowfill@@\relax\relbar\rightarrow}
\def\leftarrowfill@@{\arrowfill@@\leftarrow\relbar\relax}
\def\leftrightarrowfill@@{\arrowfill@@\leftarrow\relbar\rightarrow}
\def\arrowfill@@#1#2#3#4{%
  $\m@th\thickmuskip0mu\medmuskip\thickmuskip\thinmuskip\thickmuskip
   \relax#4#1
   \xleaders\hbox{$#4#2$}\hfill
   #3$%
}
\makeatother

\begin{document}

\author{Katrina Honigs}
\author{Pijush Pratim Sarmah}
\address{Department of Mathematics\\
Simon Fraser University\\
8888 University Drive\\
Burnaby, BC, V5A 1S6\\
Canada}

\title{Involutions of curves in abelian surfaces and their Jacobians}

\date{}

\begin{abstract}
We examine \'etale covers of genus two curves that occur in the linear system of a polarizing line bundle of type $(1,d)$ on a complex abelian surface. We give results counting fixed points of involutions on such curves as well as decomposing their Jacobians into isogenous products.
\end{abstract}

\maketitle

In \cite{EkedahlSerre}, Ekedahl and Serre
examine smooth, projective complex curves whose Jacobians are completely decomposable, that is, isogenous to a product of elliptic curves.
They ask if it is possible to find such curves of each genus, and if the genus of such curves is bounded.
Much progress has been made on this question --- see for example
\cite{paulsuth} for recent updates --- but it remains open.

In their paper, Ekedahl and Serre give examples of curves that are unramified abelian covers of genus $2$ curves $H$, where $JH$ is isogenous to a product of elliptic curves. We examine more generally the decomposability of the Jacobians of curves $C$ in abelian surfaces $A$ that are unramified abelian covers of genus~2 curves, though we limit our scope to curves in the linear system of a line bundle $L$ defining a primitive polarization  (\Cref{cons}).
We prove the following two general theorems on decomposing the Jacobians of these curves:

\begin{thm}[\Cref{cyclic}]\label{one}
Let $C$ be as in \Cref{cons} for
$X=\langle x \rangle\leq K(L)$ a cyclic group. We have the following isogeny relations:
\begin{enumerate}[(a)]
\item When $d$ is odd: $J_C\sim A\times J^2_{C/\langle -1 \rangle}$.
\item When $d$ is even:
  $J_C\sim A\times J_{C/\langle -1 \rangle}\times J_{C/\langle -1 \circ t_x\rangle}$.
\end{enumerate}  
\end{thm}

\begin{thm}[\Cref{prime}]\label{two}
Let $p$ be an odd prime and $d=p^2$.
Let $C$ be as in \Cref{cons} for
\[
(\Z/p\Z)^2\cong X\leq K(L).
\]  
Then we have the following isogeny relations, where each $C_i$ is a curve of genus $\frac{p-1}{2}$ that is a quotient of $C$ by $[-1]$ and a translation:
\[J_C \sim A \times \prod_{i=1}^{p+1} J_{C_i}^2,\quad
J_{C/\langle -1\rangle} \sim \prod_{i=1}^{p+1} J_{C_i}.\]
\end{thm}

\Cref{two} also gives a new proof of the main result in \cite{decomposable_jacobians}, which demonstrates that there exist curves whose Jacobian is isogenous to an arbitrarily high number of Jacobians of the same dimension.
We also give a class of examples where $J_C$ contains an elliptic  factor.

Our primary method for decomposing Jacobians is applying
\cite[Theorem~B]{KaniRosen} of Kani and Rosen to the automorphism groups of these curves. In order to obtain our results on Jacobian decompositions, we
prove general results on 
 the numbers of fixed points of involutions $[-1]\circ t_x$ on $C$, which control how low the genus of a quotient curve can be.

\begin{thm}[\Cref{odd}]\label{thm.odd}
  Let $d$ be odd and $C$ a curve as in \Cref{cons}. Any involution $[-1]\circ t_x$ on $C$ has $6$ fixed points.
\end{thm}

In \cite{BO}, Bor\'{o}wka and Ortega completed an explicit construction of all possible smooth hyperelliptic curves in (generic) abelian surfaces by producing a curve that is a Klein cover  of a genus $2$ curve on a
$(1,4)$-polarized abelian surface. We generalize this construction
by giving criteria for the existence of an involution $[-1]\circ t_x$ that has $12$ fixed points.

\begin{thm}[Propositions~\ref{cyclic.fix}, \ref{twelve}]
  \label{thm.even}
  Let $d$ be even and $C$ a curve as in \Cref{cons}, which is an $X$-covering of a smooth genus $2$ curve for $X\leq K(L)$.
  There is an $x\in X$ so that the involution $[-1]\circ t_x$ on $C$ has $12$ fixed points if and only if the Sylow $2$-subgroup of $X$ is not cyclic.  
\end{thm}

\subsection*{Outline}
In  \cref{sec_cons}, we begin  by presenting the curves $C$ that we will study.
In \cref{sec_lb}, we relate fixed points of involutions $[-1]\circ t_x$ on $C$ to translations of the line bundle $L$, and conclude by proving \Cref{thm.odd}.
In \cref{sec_invol}, we prove
\Cref{thm.even}. We then take the opportunity to briefly review the smooth hyperelliptic curves in $|L|$; when $d=2,4$, they arise from curves made with \Cref{cons}.
In \cref{sec.decomp}, we give our results on  Jacobian splittings,
and recover some of the results of Ekedahl and Serre using \cite[Theorem~B]{KaniRosen} and our results on automorphisms.

\subsection*{Notation}
The polarization isogeny given by $L$ is denoted by $\varphi_L:A\to A\spcheck$, $a\mapsto t_a^*L\otimes L^{-1}$ and we write its kernel as:
$$K(L):=\ker\varphi_L\cong (\Z/d\Z)^2.$$
We write the involution given by the inverse group law map on $A$ as $[-1]$.
We denote the set of fixed points of an involution $\iota$  on $C$ by  $\Fix_C(\iota)$ and their number by $\#\Fix_C(\iota)$.

\subsection*{Acknowledgements} The authors thank Jen Paulhus and  Nils Bruin for helpful conversations, as well as P.~Bor\'{o}wka, particularly for showing us how to extend our methods to some of the examples in \cref{JC4}.
K.~H.\ is supported by an NSERC Discovery Grant.

\section{Construction}\label{sec_cons}

We now construct the curves that we will analyze in this paper.
Let $A$ be an abelian surface that is 
$(1,d)$-polarized by a N\'eron--Severi class $l$.

Ample line bundles that have the same N\'eron--Severi class differ by translation. 
However, because we wish to identify the number of ramification points of specific automorphisms of curves in our exposition, we will fix a particular line bundle representative. Considering $A$ as a complex torus $V/\Lambda$, we fix a decomposition $V_1\oplus V_2$ for $l$ and select the line bundle $L$ of characteristic $0$ (see \cite[\S3.1]{BirkenhakeLange}) with respect to this decomposition.

\begin{cons}\label{cons}
  Let $X$ be a subgroup of $K(L)$ of order $d$ and consider the quotient $\pi:A\to A/X$. By \Cref{lem.pol} below, $A/X$ is principally polarized by
  the line bundle $\pi_*L=:M$. 
Furthermore by \cite[Prop.~6.5.2]{BirkenhakeLange}, we may choose 
a decomposition of $A/X$ for the N\'eron--Severi class of $M$
that is compatible with the decomposition of $A$ for $l$
(replacing the decomposition of $A$ if needed)  and so that both 
$L$ and $M$ have characteristic $0$ with respect to these decompositions. 

If $A/X$ is the Jacobian of a smooth genus $2$ curve (for instance if $A$ is simple this will be the case for any choice of $X$), then the curve
we construct is the preimage $C:=\pi^{-1}(H)$ in $|L|$. 
\end{cons}

\begin{rmk}\label{autos}
For each $x\in X$, the translation $t_x$ is an automorphism of the
curve $C$. Since $[-1]$ is an involution of $H$, it is an involution of $C$ as well. Thus $\Aut(C)$ contains a subgroup isomorphic to $\langle -1 \rangle \rtimes X$
whose elements are of the form  $[-1]\circ t_x$ and $t_x$ for $x\in X$.
\end{rmk}

\begin{lemma}\label{lem.pol}
  Let $X$ be a subgroup of $K(L)$ of order dividing $d$ and consider the quotient isogeny $\pi:A\to A/X$.
Then $A/X$ is $(1,\frac{d}{|X|})$-polarized by the pushforward of the polarization $\varphi_L$ by $\pi$.
\end{lemma}
  
\begin{proof}
First, we note $\ker(\pi)\subseteq \ker(\varphi_L)$.
In order to take the pushforward of $\varphi_L$ by $\pi$, we must check that $X$ is isotropic with respect to the commutator pairing $e^L$ of
the theta group of $L$.
We may write a decomposition of $K(L)$ with respect to this pairing as $\Z/d\Z\times \mu_d$, noting $\mu_d\simeq \Hom(\Z/d\Z,\C^*)$, with generators $1$ and $\zeta$ (see \cite[\S6.3]{BirkenhakeLange}, \cite[(8.21)]{EVM}).
Then:
\[
e^L((x,\chi),(x',\chi'))=\chi'(x)\chi(x')^{-1}.
\]  
As a subgroup of $K(L)$, $X$ is generated by some $(a,\zeta^{\alpha})$,
 $(b,\zeta^{\beta})$ of orders $d_1$, $d_2$ so that $d_1d_2=|X|$.
The pairing of these generators is $\zeta^{a\beta-b\alpha}=1$: since $d_1a$ and $d_2\beta$ are both multiples of $d$, $a\beta$ is a multiple of $d$ (as is $b\alpha)$. We may use similar reasoning to show that any two elements of $X$ pair trivially.

The pushforward polarization $\pi_*\varphi_L:A/X\to (A/X)\spcheck$ fills in the following diagram and corresponds to 
the line bundle $\pi_*L$ \cite[Corollary, p.~231]{Mumford} (moreover, $L\simeq \pi^*\pi_*L$):
\begin{equation}\label{pushforward}
  \vcenter{
\xymatrix{A\ar^{\varphi_L}[r]\ar^{\pi}[d]&A\spcheck\\
A/X\ar_{\varphi_{\pi_*L}}[r]&(A/X)\spcheck\ar^{\pi\spcheck}[u]
}    }
\end{equation}
That $\varphi_{\pi_*L}$ is a $(1,\frac{d}{|X|})$-polarization follows, for instance,  from an analysis of the kernels in the above diagram.
\end{proof}

\begin{rmk}
\Cref{lem.pol} does not hold for $X$ where $|X|>d$ since any maximal groups isotropic with respect to $e^L$ is of order $d$ \cite[p.~233, Thm.~4]{Mumford}.
\end{rmk}

\section{Symmetric linear systems and ramification of $[-1]$}\label{sec_lb}

In this section, we analyze
the number of fixed points of involutions $[-1]\circ t_x$ 
by relating them to properties of symmetric line bundles. We begin by reviewing properties of such line bundles.

Let $A$ be an  abelian surface that is 
$(1,d)$-polarized by a 
symmetric line bundle $L$. 

Since $L$ is symmetric, there is an isomorphism $[-1]^*L\simeq L$ and thus $[-1]^*$ is a linear involution on
 $H^0(A,L)$. 
 Its eigenvalues are $1$ and $-1$ and we denote the associated eigenspaces by
 $H^0(A,L)^+$ and $H^0(A,L)^-$, respectively.
 The involution $[-1]$ of $A$ restricts to an involution of
 each curve whose class is  contained in $H^0(A,L)^{\pm}$.

The dimensions of $H^0(A,L)^{\pm}$ are affected by the parity of $L$ and whether it has a symmetric theta structure (sts).

The Weil pairing on $A$ gives a non-degenerate strictly alternating form on $A[2]$ as an $\F_2$ vector space. The quadratic forms associated to this pairing form a principal homogeneous space over $A[2]$ and correspond to the symmetric line bundles with  N\'eron--Severi class $[L]$. These line bundles are called even or odd depending on the parity of the corresponding quadratic form.
We examine further specifics on determining this parity in \cref{parity}.

A symmetric theta structure 
 is an isomorphism between the extended theta and extended Heisenberg groups. Such an isomorphism exists if and only if a constraint on the dimension of 
$h^0(A,L)^+$ or $h^0(A,L)^-$
is satisfied
(maximality, see \cite[Exercise~6.10(10)]{BirkenhakeLange}).

A line bundle chosen to have characteristic $0$ with respect to a decomposition, as in \Cref{cons} has an sts and is even.

The following proposition summarizes the dimensions and base loci of $H^0(A,L)^\pm$.
The base loci are contained in $A[2]$
and we call them $A[2]^\mp$, named for the points where the
quadratic form associated to $L$
takes the values $-1$ and $+1$, respectively.
For any smooth curve $C\in H^0(A,L)^{\pm}$, $\Fix_C[-1]=A[2]^{\mp}$.

\begin{prop}[{{\cite[Prop.\ 4.7.5, Exercise~6.10(10)]{BirkenhakeLange}}}]\label{dimcount}
Let $A$ be an abelian surface that is $(1,d)$-polarized by a symmetric line bundle $L$. 

If $d$ is odd, then $L$ has an sts and the linear eigensystems of $L$ under the action of $[-1]^*$ have the following properties:
\begin{align*}
&\text{$L$ even:}\quad
& h^0(A,L)^+&=\tfrac{d+1}{2},
  &h^0(A,L)^-&=\tfrac{d-1}{2},\\
& &\#A[2]^-&=6, & \#A[2]^+&=10.\\[1ex]
&\text{$L$ odd:}\quad
& h^0(A,L)^+&=\tfrac{d-1}{2},
  &h^0(A,L)^-&=\tfrac{d+1}{2},\hspace*{1.5cm}\\
& &\#A[2]^-&=10, & \#A[2]^+&=6.
\end{align*}
If $d$ is even, we distinguish more cases, depending on whether $L$ has an sts.
\begin{align*} 
&\text{$L$ sts, even: }\quad &h^0(A,L)^+&=\tfrac{d}{2}+1,& h^0(A,L)^-&=\tfrac{d}{2}-1.\\
&&\# A[2]^-&=4,
                                        & \#A[2]^+&=12.
 \\
&\text{$L$ sts, odd: }\quad &h^0(A,L)^+&=\tfrac{d}{2}-1,& h^0(A,L)^-&=\tfrac{d}{2}+1,\\
&&\# A[2]^-&=12,
& \#A[2]^+&=4.\\
&\text{$L$ no sts: }
  \quad &h^0(A,L)^+&=\tfrac{d}{2},& h^0(A,L)^-&=\tfrac{d}{2},\\
&&\# A[2]^-&=8,
& \#A[2]^+&=8.  
\end{align*}
\end{prop}  

Given a curve $C$ as in \Cref{cons}, 
for each $x\in X$, $[-1]\circ t_x$ is an involution of $C$.
The following lemma shows that the fixed locus of $[-1]\circ t_x$ is in bijection with the fixed locus of $[-1]$ acting on a translated curve $C+y$, which lies in the linear system of $t_{-y}^*L$.
Subsequent results in this section allow
 us to determine the properties of
$t_{-y}^*L$ so that we may apply \Cref{dimcount} to find $\#\Fix_C([-1]\circ t_x)$.

\begin{lemma}\label{lem.fix}
Let $C$ be a curve contained in an abelian surface $A$.
Suppose there are $x,a\in A$ so that $[-1]\circ t_{a+x}$ is an involution of $C$.
  For any $y\in A$ such that $2y=x$,
$[-1]\circ t_a$ is an involution of the translated curve $C+y$ and
  there is a bijection of sets:
  \[
\Fix_C(-1\circ t_{x+a}) \leftrightarrow \Fix_{C+y}(-1\circ t_a).
\]
\end{lemma}

\begin{proof}
For any point $c\in C$, by our assumptions, $-c-x-a\in C$.
Any point in $C+y$ may be written as $c+y$ for some $c\in C$. Applying $[-1]\circ t_a$ gives $-c-y-a$, and we note $(-c-y-a)-y=-c-a-x\in C$, and thus $-c-y-a$ is in $C+y$.

Any point $c\in C$ is fixed by $[-1]\circ t_{x+a}$ if $c=-c-x-a$ and
$c+y$ is fixed by $[-1]\circ t_a$ if $c+y=-c-y-a$.
Since $2y=x$, these conditions are identical. Thus the bijection between these fixed loci is given by translation by $y$.
\end{proof}

\begin{lemma}\label{translate}
\begin{enumerate}[(a)]
\item For any $y\in A$, $t_y^*L$ is symmetric if and only if $2y\in K(L)$.
\item If $L$ has an sts, then
  for any $y\in A$, $t_y^*L$ has an sts if and only if $t_y^*L\simeq t_z^*L$ (i.e. $\varphi_L(y)=\varphi_L(z)$) for some $z\in A[2]$.
\end{enumerate}  
\end{lemma}

\begin{proof}
(a) Since $L$ is symmetric and $2y=x$, then $t_y^*L$ is symmetric:
\[
t_y^*t_y^*L\simeq t_x^*L\simeq L \Rightarrow
t_y^*L\simeq t_{-y}^*L \Rightarrow 
[-1]^*t_y^*L\simeq
[-1]^*t_{-y}^*L\simeq
t_{y}^*[-1]^*L\simeq
t_{y}^*L
\]    
The elements of the quotient group $\langle y \!\mid\! 2y\in K(L) \rangle/K(L)$
correspond to distinct line bundles $t_y^*L$ and
has order $|A[2]|$. Thus we obtain all $16=|A[2]|$ symmetric line bundles in this way.

(b) See the proof of \cite[Theorem~6.9.5]{BirkenhakeLange}.
Moreover, the symmetric line bundles admitting symmetric theta structures correspond to the cosets $A[2]/(K(L)\cap A[2])$. 
\end{proof}

\begin{lemma}\label{eigen}
Let $C$ be a curve produced using \Cref{cons} with the group $X\leq K(L)$.
\begin{enumerate}[(a)]
\item $[C]\in H^0(A,L)$ has eigenvalue $1$ with respect to action of $[-1]^*$.
\item Let $a\in A$ such that $2a\in X$ and let $M$ be the principal polarization of $A/X$ produced in \Cref{cons}.
  The eigenvalue of $[C-a]\in H^0(A,t_a^*L)$ is
$1$ or $-1$ if $t_{\pi(a)}^*M$ is even or odd, respectively.
\end{enumerate}
\end{lemma}  

\begin{proof}
(a) Let $C$ be a curve produced using \Cref{cons} with the subgroup   $X\leq K(L)$. Let $\pi: A\to A/X$ be the quotient isogeny. Then $M:=\pi_*L$
  principally polarizes $A/X$ and let $H\subset A/X$ be the genus $2$ curve so that $M\simeq \O(H)$.

Both $L$ and $M$ both have characteristic $0$ with respect to compatible decompositions.
 In particular, $M$ is even. Since $\pi$ commmutes with $[-1]$, and the following pullback maps $[H]$ to $[C]$, $[C]$ must also have an eigenvalue of $1$ with respect to $[-1]^*$:
\begin{equation}\label{global}
\pi^*: H^0(J_H,M)\to H^0(A,L)
\end{equation}

(b) Let $a\in A$ be such that $2a\in X$. Then
$t_{\pi(a)}^*M$ is symmetric and since $X\leq K(L)$, by
\Cref{translate}(a), $t_a^*L$ is symmetric.
The following pullback maps $[H-\pi(a)]$ to $[C-a]$ and is compatible with $[-1]^*$:
\begin{equation}
\pi^*: H^0(A/X,t_{\pi(a)}^*M)\to H^0(A,t_a^*L).
\end{equation}
Thus, the eigenvalue of $[C-a]$ under the action of $[-1]^*$ 
is $1$ if
$t_{\pi(a)}^*M$ is even and $-1$ if it is odd.
\end{proof}

Let $C$ be any curve in \Cref{cons}.
Combining \Cref{dimcount} and \Cref{lem.fix}, we may find $\#\Fix_C([-1]\circ t_x)$ by producing $y\in A$ so that $2y=x$ and determining whether $t_{-y}^*L$ has an sts and if so, determining the parity of $t_{-y}^*L$ and the eigenvalue of $[C-y]\in H^0(A,L)$ under the action of $[-1]^*$. We may determine whether $t_{-y}^*L$ has an sts using the criteria in \Cref{translate}(b) and by \Cref{eigen}, we may determine the eigenvalue of $[C-y]$ given the parity of $t_{\pi(a)}^*M$.
We use this line of reasoning to distinguish several cases in the following propositions.

\begin{prop}\label{odd}
Let $d$ be odd and $C$ be a curve as in \Cref{cons}. 
For any $x\in X$, the action of $[-1]\circ t_x$ on $C$ has $6$ fixed points.
\end{prop}

\begin{proof}
Since the order of $x$ is odd, we may choose $y\in \langle x \rangle$ of the same order as $x$ so that $2y=x$. 
Then, $y\in X \leq K(L)$, $t_{-y}^*L\simeq L$ and $\pi(-y)=0$, so $t_{\pi(y)}^*M\simeq M$.
Both $L$ and $M$ are even, so
\[\#\Fix_{C+y}[-1]=\#\Fix_C([-1]\circ t_x)=6.\hfill\qedhere\]\end{proof}

\begin{prop}\label{even}
Let $d$ be even and $C$ be a curve as in \Cref{cons}.
For any $x\in X$, one of the following mutually exclusive conditions applies to $\Fix([-1]\circ t_x)$.

Let $n$ be the order $x\in X$ and $2^m$ the highest power of $2$ dividing $d$.
\begin{enumerate}[(a)]
\item Suppose that for all $y\in A$ such that $2y=x$, $y\not\in K(L)$. Then, $\#\Fix([-1]\circ t_x)=8$.  This case applies, for instance, if $2^m|n$.
\end{enumerate}  

Now, suppose there exists $y\in A$ such that $2y=x$ and $y\in K(L)$.

\begin{enumerate}[(a)]
\item[(b)] If $t_{\pi(-y)}^*M$ is even, then $\#\Fix([-1]\circ t_x)=4$.
This case applies, for instance, if $y\in X$ (hence $t_{\pi(y)}^*M=M$), which must occur if $n$ is odd or if $n$ is even and $x$ is contained in a cyclic subgroup of $X$ with order $2n$.
\item[(c)] If $t_{\pi(-y)}^*M$ is odd, then $\#\Fix([-1]\circ t_x)=12$.
\end{enumerate}
\end{prop}

\begin{proof}
(a) Let $y\in A$ such that $2y=x$. The condition for part (a) is equivalent to the statement that for any $z\in A[2]$, $z-y\not\in K(L)$. 
By \Cref{translate}(b), $t_{-y}^*L$ does not have a symmetric theta structure, hence $\#\Fix([-1]\circ t_x)=8$.

This condition applies if $2^m|n$:
For any $y\in A$ such that $2y=x$, $y$ has order $2n$ and therefore cannot be contained in $K(L)$.

(b) Since $y\in K(L)$, $t_{-y}^*L\simeq L$ has an sts and is even, and the result follows.

If $n$ is odd, then there is some $y\in \langle x \rangle$ such that $2y=x$, so $y\in X$. If $n$ is even but $x$ is contained in a cyclic subgroup of $X$ of order $2n$ then we may take $y$ to be its generator.

(c) Under the above conditions, $t_{-y}^*L\simeq L$ has an sts and is even, so the result follows.
\end{proof}

\subsubsection{Even and odd}\label{parity}
To apply \Cref{even},
we examine how to determine the parity of $t_{\pi(-y)}^*M$ for $y\in K(L)$ such that $2y\in X$.

In \Cref{cons}, we have fixed a decomposition of $A/X$ for the N\'eron--Severi class of $M$. The induced decomposition on the $2$-torsion subgroup is a sum of maximal isotropic subgroups of $(A/X)[2]$ with respect to the Weil pairing:
\[
(A/X)[2]\cong Z_1\oplus Z_2.
\]  
This decomposition allows us to determine the parity of $t_z^*M$ for any $z\in (A/X)[2]$, which is even or odd if the quadratic form $q_M(z)$ is $1$ or $-1$, respectively. We may choose dual symplectic bases $\{u_1,u_2\}$ and $\{v_1,v_2\}$ of $Z_1$ and $Z_2$ as $\F_2$-vector spaces so that $\langle u_i,v_j \rangle=(-1)^{\delta_{ij}}$. We may write $z=u+v$ for $u\in Z_1$, $v\in Z_2$, and then $q_M(z)=\langle u,v \rangle$.

However, if $z=\pi(-y)$ for some $y\in K(L)$ (as in the cases we treat in~\Cref{even}), then it is contained 
in a smaller subspace of $(A/X)[2]$. The decomposition of $A$ for $l$ induces a decomposition of $K(L)$ as follows:
\[
K(L)=K_1\oplus K_2=\langle k_1 \rangle \oplus \langle k_2 \rangle.
\]
The image of $K(L)$ in $A/X$ is
$\pi(K(L))=\langle \pi(k_1) \rangle\oplus \langle \pi(k_2)\rangle$, which, by the compatibility assumptions of \Cref{cons}, respects the decomposition of $A/X$ with respect to $M$.
Since $\pi(K(L))$ is generated by two cyclic groups, its intersection with $(A/X)[2]$ (if nontrivial) is contained in two cyclic groups:
\[
\pi(K(L))\cap (A/X)[2] \subseteq \langle w_1 \rangle \oplus \langle w_2 \rangle
\]  
for some $w_i$ such that $\Z/2\Z\cong \langle w_i \rangle\leq Z_i$ and
$\langle w_1,w_2 \rangle=-1$.
Then we have the following result:
\begin{align*}
&\text{If $\pi(-y)$ is $0$, $w_1$, or $w_2$, then $t^*_{\pi(-y)}M$ is even.}
  \\
&\text{If $\pi(-y)$ is $w_1+w_2$, then $t^*_{\pi(-y)}M$ is odd.}  
\end{align*}

\section{Counting fixed points of $[-1]\circ t_x$}\label{sec_invol}

In this section, let $C$ be as in \Cref{cons} and $d$ be even. We give further results on $\#\Fix_C([-1]\circ t_x)$ for $x\in X$, particularly investigating the unusual case where the number of fixed points is $12$,  and then apply these results to review the number of smooth hyperelliptic curves in $L$.

\subsection{When can $\#\Fix_C([-1]\circ t_x)=12$?}
In the first result in this section, we see a situation where such an involution cannot occur, and in our second result, we characterize such involutions. Let $2^m$ be the highest power of $2$ dividing $d$. Let $X_2$ be the Sylow $2$-subgroup of $X$. We refer to the projection of any $x\in X$ to $X_2$ as the $2$-primary part of $x$.

\begin{prop}\label{cyclic.fix}
If $X_2$ is cyclic, then, for any $x\in X$:
\begin{enumerate}[(a)]
\item $\#\Fix_C([-1]\circ t_x)=8$ if the $2$-primary part of $x$ generates $X_2$.
\item $\#\Fix_C([-1]\circ t_x)=4$ if the $2$-primary part of $x$ does not generate $X_2$.
\end{enumerate}  
\end{prop}

\begin{proof}
  The order of $X_2$ is $2^m$, which is also the highest order of any elements in the $2$-primary part of $K(L)$.

(a) If the $2$-primary part of $x$ generates $X_2$, we may choose $y$ such that $2y=x$ where $y\not\in K(L)$ (we may choose the non $2$-primary part of $y$ to be in $X$). Thus \Cref{even}(a) applies.

(b) Choose $y\in X$ so that $2y=x$ and apply \Cref{even}(b).  
\end{proof}

\begin{prop}\label{twelve}
  If $X_2$ is not cyclic, then there is an $x\in X$ so that $\#\Fix_C([-1]\circ t_x)=12$.

Moreover, for any $x\in X$, $\#\Fix_C([-1]\circ t_x)$ is either $4$ or $12$.
\end{prop}

\begin{proof}
We may write the decomposition of the $2$-primary part of $K(L)$ for~$l$ as follows:
  \[
K(L)_2:=K(L)\cap A[2^m]\cong  K_1 \oplus K_2 =
    \langle k_1 \rangle \oplus \langle k_2\rangle.
  \]
Since we have chosen compatible decompositions in \Cref{cons},
we may also write $X_2$  as a sum with respect to this decomposition for some $0<m_1\leq m_2$ and $m_1+m_2=m$:
\[X_2\simeq\Z/2^{m_1}\Z \oplus \Z/2^{m_2}\Z .\]
Since $m_1,m_2<m$, for any $x\in X$, there exists $y\in K(L)$ so that $2y=x$,
and thus \Cref{even}(a) cannot apply and $\#\Fix_C([-1]\circ t_x)$ cannot be $8$, leaving $4$ and $12$ as the possibile numbers of fixed points.

Again, by examining orders of elements, there is
a $y\in K(L)\setminus X$ so that $2y=x\in X$ and
$\pi(y)$ is not contained in either group
$\langle \pi(k_i) \rangle$.
As shown in \cref{parity}, $t_{\pi(-y)}^*$ must be odd. 
\end{proof}  
  
\subsection{Smooth hyperelliptic curves in $(1,d)$-polarizations}

In this section, we review the number of smooth hyperelliptic curves in a fixed symmetric linear system of type $(1,d)$ on a general abelian surface.
These values are nonzero for $1\leq d\leq 4$ and in these cases are $1$, $6$, $9$ and $4$, respectively (see \cite[Table~1]{Bryan}).

These curves
are symmetric and have hyperelliptic involution $[-1]$ or $[-1]\circ t_x$ for some $x\in K(L)$ or are translations of such a curve:  by \Cref{lem.fix}, for any curve $C\subset A$ with involution $[-1]\circ t_x$ and $b\in K(L)$, $\#\Fix_C([-1]\circ t_x)=\#\Fix_{C+b}([-1]\circ t_{x-2b})$.

Smooth curves in $|L|$ have genus $d+1$ and thus a hyperelliptic involution on such a curve must have $2d+4$ fixed points.
For $d=2$ or $4$, we produce these curves using \Cref{cons}. In those cases, we may then produce $|K(L)/X|$ distinct translations of each curve.

\subsubsection{$d$ odd}

In the principally polarized case ($d=1$), the abelian surface is the Jacobian of a genus $2$ curve. 
Any symmetric polarizing line bundle on such a surface has its linear system given by (a translation of) this curve with hyperellipitic involution $[-1]$, making the total number of hyperelliptic curves in a fixed linear system
$|K(L)|=1$.

When $d=3$,  a hyperelliptic involution must have $10$ fixed points. In \Cref{dimcount},
we see that the systems $H^0(A,L)^{\pm}$
contain a single smooth curve where $[-1]$ has $10$ fixed points.
By \Cref{odd}, this curve cannot occur 
 as a cover of a genus $2$ curve using \Cref{cons}.
Bor\'{o}wka and Sankaran \cite{BS} exhibit
it in terms of theta functions and, show that,
up to translation, it is the unique hyperelliptic curve in a general
$(1,3)$-polarized
abelian surface.
 
Given this curve $C$, 
for each $b\in K(L)$, the translation $C+b$ is  a distinct hyperelliptic curve in $|L|$. The total number of hyperelliptic curves in $|L|$
produced in this way is:
\[|K(L)|=9.\]

\subsubsection{$(1,2)$-polarizations}\label{hyper.two} In this case, a hyperelliptic involution must have $8$ fixed points.
Any curve $C$ as in \Cref{cons} is formed using an order two cyclic group
$\langle x \rangle=X\leq K(L)$.
By \Cref{cyclic.fix}, $\#\Fix_C[-1]=4$ and $\#\Fix_C([-1]\circ t_x)=8$, making $[-1]\circ t_x$ a hyperelliptic involution.

A detailed treatment of the $(1,2)$-polarized case where $L$ has a symmetric theta structure is given by Barth in  \cite{Barth}. Our argument here recovers his result that, for instance,
a hyperelliptic involution of the curves in $|L|$ must be distinct from $[-1]$.

The total number of hyperelliptic curves
in $|L|$ that
we can produce from this construction, as well as translating, is thus:
\[
\sum_{\Z/2\Z\simeq X\leq K(L)} |K(L)/X|=3\cdot 2=6.
\]

\subsubsection{$(1,4)$-polarizations}\label{14invol}
In this case a hyperelliptic involution must have $12$ fixed points.
Curves can be obtained using \Cref{cons} using a choice of order $4$ subgroup
$X\leq K(L)\simeq (\Z/4\Z)^2$. 

By \Cref{cyclic.fix}, if $X\leq K(L)$ is cyclic, then
$\#\Fix([-1]\circ t_x)$ is $4$ or $8$, so 
\Cref{cons} will not produce a hyperelliptic involution.

Now, let $X\leq K(L)$ be the Klein group. We may write $K(L)$ as a sum with respect to the decomposition of $A$ for $[L]$:
\[
\langle k_1 \rangle \oplus  \langle k_2 \rangle \cong K_1 \oplus K_2.
\]  
Using this notation, $X=\langle 2k_1 \rangle \oplus \langle 2k_2 \rangle$.
By \Cref{twelve}, 
\begin{gather*}
\#\Fix_C[-1]=\#\Fix_C([-1]\circ t_{2k_1})=\#\Fix_C([-1]\circ t_{2k_2})=4,\\
\#\Fix_C([-1]\circ t_{2k_1+2k_2})=12,
\end{gather*}
making $[-1]\circ t_{2k_1+2k_2}$ a hyperelliptic involution for $C$.

Since there is only one way to embed the Klein group $V_4$ into $K(L)$,
we have found one hyperelliptic curve using \Cref{cons}.
The total 
number of hyperelliptic curves in $|L|$ that we may then produce by  translating is:
\[
|K(L)/V_4|=4.
\]

\begin{rmk}
This construction coincides with that of Bor\'{o}wka and Ortega \cite{BO},
and so we may think of \Cref{even} as a generalization of their results. However, in \Cref{cons} we have specified additional data by choosing decompositions.

The authors prove their curve is the unique hyperelliptic curve (up to translation) on a general $(1,4)$-polarized abelian surface.
It is constructed as a cover of a genus $2$ curve, $\widetilde{C}\to H$ with respect to the Klein group $\{0,\eta_1,\eta_2,\eta_1+\eta_2\}=G\leq J_H[2]$.
The result \cite[Theorem~4.7]{BO} states that if $\widetilde{C}$ is hyperelliptic then $G$ is non-isotropic with respect to the Weil pairing on $J_H[2]$ and that, conversely, if $G$ is non-isotropic and $\eta_1$ and $\eta_2$ may be written as the difference of two Weierestrass points, then $\widetilde{C}$ is hyperelliptic.

In the notation of this paper $\pi(K(L))$ plays the role of the group $G$ generated by  $\pi(k_1)$, $\pi(k_2)$. It's precisely this non-isotropy condition that allows us to pick an element of $z\in\pi(K(L))$ where $t_{z}^*M$ is odd. 
The Weierstrass points of $z\in H$ are those where $q_M(z)=-1$,
but since we have arranged that  $q_M(\pi(k_1))=q_M(\pi(k_2))=1$, we can see that $\pi(k_1)$ and $\pi(k_2)$ must be differences of Weierstrass points. 
\end{rmk}  

\section{Decomposing Jacobians using subgroups}\label{sec.decomp}

By Poincar\'e reducibility, every abelian variety is isogenous to a product of simple abelian varieties that are unique, up to isogeny.
In this section, we decompose the Jacobians of the curves $C$ produced using \Cref{cons} into, up to isogeny, products of smaller abelian varieties.

In \cref{decomp.general}, 
we give some general consequences of  \cite[Theorem~B]{KaniRosen}
for decomposing $J_C$. In \cref{JC4}, we use our methods to examine several cases where $J_C$ is isogenous to a product of $A$ and elliptic curves, recovering some results of \cite{EkedahlSerre}.
Finally in \cref{cover} we give some examples of curves that cover an elliptic curve.

\subsection{General results on decomposing Jacobians}\label{decomp.general}

Kani and Rosen prov{\-}ed the following theorem by showing that partitions of the automorphism group of $C$ induce idempotent relations in the endomorphism algebra of the Jacobian $J_C$.
The automorphisms of the curves $C$ produced using \Cref{cons} contain the group $\langle -1 \rangle \rtimes X$ (\Cref{autos}),
but not every such group admits a partition.
The only subgroups admitting a partition are
dihedral groups or certain $p$-groups (see \cite[Theorem~3.5.10]{Schmidt}), 
hence
the structure of this group affects the results that may be gained from applying this theorem. 

\begin{thm}[{{\cite[Theorem~B]{KaniRosen}}}]\label{KR}
  Let $C$ be a (smooth, projective, geometrically connected) curve and $G$ be a finite subgroup of $\Aut(C)$ such that $G=H_1\cup\cdots\cup H_t$ where the subgroups $H_i\leq G$ satisfy $H_i\cap H_j=1$ if $i\neq j$. Then we have the isogeny relation
  \begin{equation}
J_C^{t-1}\times J_{C/G}^g\sim J_{C/H_1}^{h_1}\times\cdots\times J_{C/H_t}^{h_t}
  \end{equation}    
  where $|G|=g$ and $|H_i|=h_i$.
\end{thm}  

We deduce two general results on the decomposition of Jacobians of curves $C$ produced as in \Cref{cons} in the cases where $X\leq K(L)$ is cyclic or it is a product of two cyclic groups of odd prime order.

\begin{thm}\label{cyclic}
Let $C$ be as in \Cref{cons} for
$X=\langle x \rangle\leq K(L)$ a cyclic group (of order $d$). We have the following isogeny relations:
\begin{enumerate}[(a)]
\item When $d$ is odd: $J_C\sim A\times J^2_{C/\langle -1 \rangle}$.
\item When $d$ is even:
  $J_C\sim A\times J_{C/\langle -1 \rangle}\times J_{C/\langle -1 \circ t_x\rangle}$.
\end{enumerate}  
\end{thm}

\begin{proof}
  $\Aut(C)$ contains the dihedral group $G:=\langle -1, t_x\rangle$ of order $2d$, which admits the following partition:
\[
 G=\langle t_x\rangle\cup \langle -1\rangle \cup
 \langle -1\circ t_x \rangle\cup\cdots\cup \langle-1\circ t_{(d-1)x} \rangle.
\]
Applying \Cref{KR} gives the following isogeny relation:
\[
  J_C^d\times J_{C/G}^{2d}\sim 
  J_{C/\langle t_x\rangle}^d\times
J_{C/\langle -1\rangle}^2\times J_{C/\langle -1\circ t_x \rangle}^2\times\cdots\times
J_{C/\langle-1\circ t_{(d-1)x} \rangle}^2.
\]
We now analyze these quotients further. The quotient $C/G$ is isomorphic to the quotient of the genus two curve $C/X\cong H$ of \Cref{cons} by its hyperelliptic involution, and thus $C/G\cong \P^1$, making $J_{C/G}$ trivial.

Furthermore, quotienting a curve by conjugate subgroups yields isomorphic quotient curves.
When $d$ is odd, the order two subgroups of $G$ are precisely the $2$-Sylow subgroups of $G$ and hence are conjugate, proving (a).

When $d$ is even, the order two subgroups generated by
$-1, -1\circ t_{2x}, \ldots, -1\circ t_{(d-2)x}$ are conjugate to one another, and the the remaining order two subgroups are all conjugate to $\langle -1\circ t_x \rangle$, proving result (b).
\end{proof}

\begin{thm}\label{prime}
Let $p$ be an odd prime and $d=p^2$.
Let $C$ be as in \Cref{cons} for:
\[
(\Z/p\Z)^2\cong X=\langle x_1 \rangle \times \langle x_2 \rangle\leq K(L).
\]  
Then we have the following isogeny relations, where each $C_i$ is a curve of genus $\frac{p-1}{2}$ that is a quotient of $C$ by $[-1]$ and a translation:
\[J_C \sim A \times \prod_{i=1}^{p+1} J_{C_i}^2,\quad
J_{C/\langle -1\rangle} \sim \prod_{i=1}^{p+1} J_{C_i}.\]
\end{thm}

\begin{proof}
  The subgroup $G:=\langle t_{x_1}, t_{x_2}\rangle$ of $\Aut(C)$ has the following partition:
  \[
    \langle t_{x_1}\rangle\cup
    \langle t_{x_2}\rangle\cup
    \langle t_{x_1+x_2}\rangle\cup
\langle t_{x_1+2x_2}\rangle\cup    
\cdots\cup
\langle t_{x_1+(p-1)x_2}\rangle
.    \]
Then \Cref{KR} gives the following isogeny relation:
\begin{equation}\label{power}
J_C^p\times J_{C/G}^{p^2}\sim J_{C/\langle t_{x_1}\rangle}\times
\cdots \times
J_{C/\langle t_{x_1+(p-1)x_2}\rangle}.
\end{equation}
Each curve on the right-hand side in the above isogeny relation is a quotient of $C$ by a cyclic group of order $p$, which is in turn a cyclic cover of degree $p$ of a genus $2$ curve. 

Since the quotient $C/\langle t_x \rangle$ is in the linear system of a $(1,p)$-polarization on $A/\langle x_1 \rangle$, we may apply the argument used in
\Cref{cyclic}(a):
$$J_{C/\langle t_{x_1}\rangle}\sim A/ \langle x_1 \rangle\times J_{C_1}^2\sim
A \times J_{C_1}^2$$
where $C_1=C/\langle t_{px_1},-1 \rangle$ is of genus $\frac{p-1}{2}$.
Performing an analogous decomposition for each cyclic quotient of $C$ and observing that $J_{C/G}\sim A$, we simplify \eqref{power} to
\begin{equation}\label{first.decomp}
  J_C^p\sim A^p\times \prod_{i=1}^{p+1} J_{C_i}^{2p}
\ \ \Rightarrow \ \ 
  J_C\sim A\times \prod_{i=1}^{p+1} J_{C_i}^{2}.
\end{equation}

Now, let $G:=\langle t_{x_1},t_{x_2},[-1] \rangle \leq \Aut(C)$, which has the following partition:
\[
  G= \biggl(\bigcup_{0\leq m,n\leq p-1} \langle [-1]\circ t_{mx_1+nx_2} \rangle\biggr)
  \cup \langle t_{x_1},t_{x_2} \rangle.
  \]
  Applying \Cref{KR} gives the isogeny relation:
\begin{equation}\label{ir}
    J_C^{p^2}\times J_{C/G}^{p^2}\sim J_{C/\langle t_{x_1},t_{x_2} \rangle}^{p^2}
    \times
\prod_{0\leq m,n\leq p-1} J^2_{C/\langle -1\circ t_{mx_1+nx_2}\rangle}.
\end{equation}
Now, in the notation of \Cref{cons}, $J_{C/\langle t_{x_1},t_{x_2} \rangle}\cong J_H\sim A$, and hence $C/G\cong \P^1$. All order $2$ subroups of $G$ are Sylow, hence conjugate to $\langle -1 \rangle$.
Simplifying  \eqref{ir} and combining with \eqref{first.decomp} yields:
\[
J_{C/\langle -1 \rangle}\sim\prod_{i=1}^{p+1}J_{C_i}.\qedhere
\]  
\end{proof}

\subsection{Completely decomposable Jacobians}\label{JC4}

In \cite[\S4]{EkedahlSerre}, Ekedahl and Serre prove the following result using characters. In this section, we examine some cases where this theorem applies using \Cref{KR} and our results on automorphisms and ramification.

\begin{thm}[{{\cite[Prop.~3 and Corollaire]{EkedahlSerre}}}]
Let $H$ be a curve of genus~$2$ and  
let $C\to H$ be a finite, unramified abelian covering of curves whose Galois group has exponent dividing $4$ or $6$. If $J_H$ is isogenous to a product of elliptic curves, then $J_C$ is completely decomposable.
\end{thm}

This result implies that if the group $X$ of \Cref{cons} has exponent dividing $4$ or $6$ and $A$ is isogenous to a product of elliptic curves, then $J_C$ is completely decomposable. We now examine the cases where $X$ is cyclic as well as the smallest non-cyclic case, the Klein cover of \cite{BO}.

\subsubsection*{$X$ cyclic}
When $d=2$ or $3$, \Cref{cyclic} directly implies complete decomposability of $J_C$.
When $d=2$, as shown in \cref{hyper.two},
$J_{C/\langle -1 \rangle}$ is elliptic and $J_{C/\langle -1 \circ t_x\rangle}$ is trivial.
When $d=3$,  by \Cref{odd}, $J_{C/\langle -1 \rangle}$ is an elliptic curve (cf.~\cite[\S5]{BLmoduli}).

When $d=4$ and $X\cong \Z/4\Z$, \Cref{cyclic} and \Cref{cyclic.fix} 
imply that $J_C$ may be decomposed into the product of $A$ with a surface 
$J_{C/\langle -1 \rangle}$ and an elliptic curve $J_{C/\langle -1 \circ t_x\rangle}$.
However, ${C/\langle t_{2x} \rangle}$ is a curve in the linear system of a $(1,2)$-polarization and $J_{C/\langle t_{2x},-1 \rangle}$ is an elliptic curve. 
Since ${C/\langle -1 \rangle}$ covers an elliptic curve,
$J_{C/\langle -1 \rangle}$ is isogenous to a product of elliptic curves.

When $d=6$ and $X\cong \Z/6\Z$, \Cref{cyclic} and \Cref{cyclic.fix} 
imply that $J_C$ may be decomposed into the product of $A$ with a threefold 
$J_{C/\langle -1 \rangle}$ and a surface $J_{C/\langle -1 \circ t_x\rangle}$.
Similarly to the previous case, $C/\langle t_{3x} \rangle$ is in a $(1,3)$-polarization and $J_{C/\langle t_{3x},-1 \rangle}$ is an elliptic curve.
Also, $C/\langle t_{2x} \rangle$ is in a $(1,2)$-polarization and 
$J_{C/\langle t_{2x},-1 \rangle}$ is an elliptic curve.
Since ${C/\langle -1 \circ t_{3x}\rangle}$ covers ${C/\langle t_{3x},-1 \rangle}$ and $J_{C/\langle -1 \circ t_x\rangle}\cong J_{C/\langle -1 \circ t_{3x}\rangle}$,
$J_{C/\langle -1 \circ t_x\rangle}$ is isogenous to a product of elliptic curves.
The threefold $J_{C/\langle -1 \rangle}\cong J_{C/\langle -1\circ t_{2x} \rangle}$
covers the distinct elliptic curves 
$J_{C/\langle t_{2x},-1 \rangle}$ and $J_{C/\langle -1 \circ t_{3x}\rangle}$, so is completely decomposable.

\subsubsection*{Klein cover}

Let $d=4$ and $X$ be the Klein group. We may write
$X=\langle x_1 \rangle\oplus\langle x_2 \rangle
\cong (\Z/2\Z)^2$. In this notation, $[-1]\circ t_{x_1+x_2}$ is the hyperelliptic involution (see \cref{14invol}).
Analogously to the proof of \Cref{prime},
the group $G:=\langle t_{x_1},t_{x_2} \rangle\leq \Aut(C)$ admits a partition:
$$G=\langle t_{x_1} \rangle \cup \langle t_{x_2} \rangle \cup \langle t_{x_1+x_2} \rangle.$$
The isogeny relation given by \Cref{KR} simplifies to:
\begin{equation}\label{isog}
  J_C\times A^2\sim J_{C/\langle t_{x_1} \rangle}\times J_{C/\langle t_{x_2} \rangle}\times
J_{C/\langle t_{x_1+x_2} \rangle}.
\end{equation}
There is a cover $C/\langle t_{x_1} \rangle\to H$ given by taking the quotient by $t_{x_2}$. Thus the automorphism group of $C/\langle t_{x_1} \rangle$ contains the order $4$ dihedral group $\langle -1, t_{x_2} \rangle$.
Applying \Cref{KR}, we have:
\[
  J_{C/\langle t_{x_1} \rangle}\sim A/\langle {x_1} \rangle \times J_{C/\langle t_{x_1},-1 \rangle}
\times J_{C/\langle t_{x_1},-1\circ t_{x_2} \rangle}.
\]
Since $\langle t_{x_1}, -1\circ t_{x_2} \rangle$ contains the hyperelliptic involution,
$J_{C/\langle t_{x_1},-1\circ t_{x_2} \rangle}$ is trivial, meaning that 
$J_{C/\langle t_{x_1},-1 \rangle}$ is elliptic.
Arguing similarly
for $C/\langle t_{x_2} \rangle$ and
$C/\langle t_{x_1+x_2} \rangle$, the isogeny relation \eqref{isog} reduces to the following relation between $J_C$ and the product of $A$ with three elliptic curves (see \cite[Theorem~5.5]{BO}), which are each the Jacobian of a quotient of $C$ by a subgroup of $\Aut(C)$ that does not contain its hyperelliptic involution:
\[
  J_C\sim A\times J_{C/\langle t_{x_1},-1 \rangle}\times
J_{C/\langle t_{x_2},-1 \rangle}
\times
J_{C/\langle t_{x_1+x_2}, -1\circ t_{x_1} \rangle}
.
\]

Finally, we note that \Cref{prime} directly implies complete decompsability in the case where $X=(\Z/3\Z)^2$ (see \cite{borowkashatsila} for analysis of this example).

\subsection{Curves that cover an elliptic curve}\label{cover}

We can use our methods to examine the following weaker condition:
does the curve $C$ of \Cref{cons} cover an elliptic curve? The following result gives some cases where this question has a positive answer.

\begin{prop}
When $d=2k$, $3k$, or $4k$ for some $k\in \N$ and $X\cong \Z/j\Z\times \Z/k\Z$ for $j=2$, $3$ or $4$, the curve $C$ of \Cref{cons} covers an elliptic curve. Further the degree of the cover is $d$ and $J_C$ has an elliptic factor in its isogeny class.
\end{prop}

\begin{proof}
  Let $d=2k$ and let $x_1,x_2\in K(L)$ be elements of orders $k$ and $2$ so that
  $X=\langle x_1,x_2 \rangle\leq K(L)$ is a subgroup of order $2k$.

  The curve $\widetilde{C}:=C/\langle t_{x_1} \rangle$ has genus $3$ and quotienting by $t_{x_2}$ gives a cover $\widetilde{C}\to H$, and thus the automorphism group of $\widetilde{C}$ contains $\langle t_{x_2},-1 \rangle\cong D_4$. Applying
  \Cref{KR}, we have:
  \[
J_{\widetilde{C}}\sim A/\langle x_1 \rangle \times J_{\widetilde{C}/\langle -1 \rangle}\times J_{\widetilde{C}/\langle -1\circ t_{x_2} \rangle}.
\]
Comparing dimensions, we see that one of
$\widetilde{C}/\langle -1 \rangle$ or
$\widetilde{C}/\langle -1\circ t_{x_2} \rangle$ must be an elliptic curve $E$ and the covering map $C\to E$ has degree $2k$.

Since $J_C\sim J_{\widetilde{C}}\times P(C/\widetilde{C})$, where
$P(C/\widetilde{C})$ is the Prym variety, $E$ appears as an isogenous factor in the decomposition of $J_C$.

For $d=3k$, we make the same argument with $x_2$ an element of order~$3$. Then $\widetilde{C}:=C/\langle t_{x_1} \rangle$ is a curve of genus $4$ whose automorphism group contains $\langle -1, t_{x_2} \rangle\cong D_6$ and
$J_{\widetilde{C}}\sim A/\langle x_1 \rangle\times J_{C/\langle -1 \rangle}^2$, thus
$\widetilde{C}/\langle -1 \rangle =:E$ is elliptic, $C\to E$ has degree $3k$, and $E$ is an isogenous factor of $J_C$.

For $d=4k$, we let $x_2$ have order $4$.
The curve $\widetilde{C}:=C/\langle t_{x_1} \rangle$ has genus~$5$ and its automorphism group contains $D_8$. The Jacobian $J_{\widetilde{C}}$ thus also has an elliptic factor $E$ where $C\to E$ has degree $4k$.
\end{proof}

\bibliographystyle{alpha}
\bibliography{main}

\end{document}